\documentclass[10pt,twoside]{amsart}
\usepackage{amsmath, amsthm, amscd, amsfonts, amssymb, graphicx, color, xcolor, soul }
\usepackage[bookmarksnumbered, colorlinks, plainpages]{hyperref}
\usepackage{graphicx}
\usepackage{epstopdf}

\sethlcolor{green}
\newtheorem{thm}{Theorem}[section]
\newtheorem{cor}[thm]{Corollary}
\newtheorem{lem}[thm]{Lemma}
\newtheorem{prop}[thm]{Proposition}
\newtheorem{defn}[thm]{Definition}

\newtheorem{ex}[thm]{Example}

\newtheorem*{thmA}{Theorem A}

\newtheorem*{thmB}{Theorem B}

\newtheorem{procedure}[thm]{}

\numberwithin{equation}{section}
\def\pn{\par\noindent}

\begin{document}

\vspace{1.3 cm}
\title{Graph homomorphisms and components of quotient graphs  }
\author{Daniela Bubboloni
}

\thanks{{\scriptsize
\hskip -0.4 true cm MSC(2010):  05C60, 05C70, 05C40.
}}
\maketitle


\begin{abstract}

We study how the number $c(X)$ of components of a graph $X$ can be expressed through the number and properties of the components of a quotient graph $X/\hspace{-1mm}\sim.$
We partially rely on classic qualifications of graph homomorphisms such as locally constrained homomorphisms and on the concept of equitable partition and orbit partition. We introduce the new definitions of pseudo-covering homomorphism and of component equitable partition, exhibiting interesting  inclusions among the various classes of considered homomorphisms. As a consequence, we find a procedure for computing $c(X)$ when the projection on the quotient   $X/\hspace{-1mm}\sim$ is pseudo-covering. That procedure becomes particularly easy to handle when the partition corresponding to $X/\hspace{-1mm}\sim$ is an orbit partition.

\end{abstract}

\vskip 0.2 true cm
\pagestyle{myheadings}
\markboth{\rightline {\scriptsize  Bubboloni}}
         {\leftline{\scriptsize Graphs homomorphisms }}

\bigskip
\bigskip
\section{\bf Introduction and main results}
\vskip 0.4 true cm

In algebra it is very common to study the properties of a set, endowed with some structure, by its quotients. Passing to a quotient reduces the complexity and allows one to focus only on certain properties, disregarding inessential details.
That idea has revealed  to be immensely fruitful especially in group theory, where manageable theorems describe the link between group homomorphisms and quotient groups.
In graph theory the notion of quotient graph appears less natural to deal with (\cite{hahn}, \cite{kna}).  That depends in large part on the fact that no notion of kernel is possible for a graph homomorphism. As a consequence it is often difficult to understand which properties are preserved in passing from a graph to a quotient graph. In this paper, developing a theory of graph homomorphisms, we show how to use information on the quotient graph components to get information on the graph components.
All the considered graphs are finite, undirected, simple and reflexive, that is, they have a loop on each vertex. Reflexivity  simplifies the study of graph homomorphisms without affecting connectivity.
Let $X$ and $Y$ be two graphs and let  $\varphi$ be a homomorphism from $X$ to $Y.$
Recall that  $\varphi$  is called complete if it maps both the vertices and edges of $X$ onto those of $Y$. Our starting point is that dealing with the quotients of a given graph is equivalent to dealing with the complete homomorphisms from it to any possible target graph (Lemmata \ref{nat-comp} and \ref{h}).
Unfortunately, $\varphi$ being  complete does not guarantee the image of a component of $X$ being a component of $Y.$ In our opinion, the property which we call ``the natural migration of the components'', is mandatory in order to control the number of components of $X$ by means of those in $Y.$  A first type of homomorphisms for which the components naturally migrate is given by those $\varphi$ which we call {\it tame}, for which vertices with the same image are connected. In that  case the number of components of $X$ and $Y$ is the same (Sections \ref{general} and \ref{hom-pa}). Moreover, there exist classic qualifications of graph homomorphisms which fit well. Recall that $\varphi$ is called locally surjective if it maps the neighborhood of each vertex  of $X$ onto the neighborhood in $Y$ of its image. Locally surjective homomorphisms have a long history in the scientific literature. Everett and Borgatti \cite{Ev} introduced them, with the name of  role colorings, for the analysis of social behavior. Recently this class of homomorphisms has received a lot of attention in theoretical computer science (\cite{FI, Ler}).

We state our main results after establishing some notation. Denote by $V_X$ the vertex set of $X$ and by $E_X$ its edge set; by $C_{X}(x)$ the component  containing $x\in V_X$; by $\mathcal{C}(X)$ the set of components and by $c(X)$ their number.  For every $C'\in \mathcal{C}(Y)$, set $\mathcal{C}(X) _{C'}=\{C\in  \mathcal{C}(X)\ :\  \varphi(C)\subseteq C' \}$; for every $y\in V_Y$ and
$\hat{X}=(V_{\hat{X}},E_{\hat{X}})$ subgraph of $X$,  put $k_{\hat{X}}(y)=|V_{\hat{X}}\cap \varphi^{-1}(y) |$,
$\mathcal{C}(X)_{y}=\{C\in\mathcal{C}(X)\ :\  k_C(y)>0\}|$ and $c(X)_{y}=|\mathcal{C}(X)_{y}|$. Denote by $\sim_{\varphi}$ the equivalence relation induced by $\varphi$ and, for every $x\in V_X$, by $C_{X}(x)/\hspace{-1mm}\sim_{\varphi}$ the quotient graph of $C_{X}(x)$ with respect to $\sim_{\varphi}$.

\begin{thmA}\label{loc-sur-com} Let $X$, $Y$ be graphs and $\varphi:X \rightarrow Y$ be a locally surjective homomorphism.
 \begin{itemize}
\item[(i)] If $C\in \mathcal{C}(X)$, then $\varphi(C)\in \mathcal{C}(Y)$. In particular, the image of $X$ is a union of components of $Y$.
\item[(ii)] For every $x\in V_X$, $\varphi(C_{X}(x))=C_{Y}(\varphi(x))\cong C_{X}(x)/\hspace{-1mm}\sim_{\varphi}.$
\item[(iii)] For every $C\in\mathcal{C}(X)$,  $\varphi^{-1}(\varphi(V_C))=\displaystyle{\bigcup_{\hat{C}\in  \mathcal{C}(X)_{\varphi(C)}}V_{\hat{C}}}$.
\item[(iv)] For $1\leq i\leq c(Y)$, let $y_i\in V_Y$ be such that $\mathcal{C}(Y)=\{C_Y(y_i): 1\leq i\leq c(Y)\}$. Then
 \begin{equation}\label{formula}
 c(X)=\sum_{i=1}^{c(Y)}c(X)_{y_i}
 \end{equation}
\end{itemize}
\end{thmA}

While the numbers $c(X)_{y_i}$ in Formula \eqref{formula} are generally difficult to compute explicitly, in a number of applications the following property gives a more manageable formula. We call $\varphi$ {\it component equitable} whenever, for every $y\in V_Y$, every component in $\mathcal{C}(X)_{y}$ intersects  the fibre $\varphi^{-1}(y)$ in a set of the same size (Section \ref{equi}). Component equitability transfers the well known notion of an equitable partition (\cite[Section 5.1]{god2}) to components rather than neighborhoods. If $\varphi$ is both locally surjective and component equitable, then $c(X)_{y_i}=k_{X}(y_i)/k_{C_i}(y_i),$ where $k_{X}(y_i)=|\varphi^{-1}(y_i)|$ and $C_i\in \mathcal{C}(X)_{y_i}$ (Proposition \ref{union2}).

The most important subset of component equitable homomorphisms is given by the so called {\it orbit homomorphisms}, that is, those homomorphisms $\varphi$  for which the equivalence classes of $\sim_{\varphi}$ in $V_X$ coincide with the orbits of a suitable group of graph automorphisms of $X$. Since the complete orbit homomorphisms are necessarily locally surjective (Proposition \ref{union2cons}), as a consequence of Theorem A, we get the following important result.

\begin{thmB}\label{summary} Let $X$, $Y$ be graphs and $\varphi:X \rightarrow Y$ be a complete orbit homomorphism. For $1\leq i\leq c(Y)$,
let $y_i\in V_Y$ be such that $\mathcal{C}(Y)=\{C_Y(y_i): 1\leq i\leq c(Y)\}$ and $C_i\in \mathcal{C}(X)_{y_i}.$ Then
\begin{equation}\label{best-for}
c(X)=\sum_{i=1}^{c(Y)}\frac{k_{X}(y_i)}{k_{C_i}(y_i)}
\end{equation}
\end{thmB}

We exhibit a precise algorithmic procedure (Procedure \ref{procedure}) for the computation of Formula \eqref{best-for}.
Moreover, we give some results to control the isomorphism class and the properties of the components (Corollaries \ref{isolated} and \ref{connection}, Proposition \ref{union2cons}\,(i) and Section \ref{iso-class}).

One of the motivations of our research is to produce a rigorous method to count the components of the proper power graph of a finite group $G$ through the knowledge of the components of some of its quotients.  Recall that the
 power graph of $G$ is the graph $P (G)$ with $V_{P (G)}=G$ and $\{x,y\}\in E_{P (G)}$, for $x,y\in G$, if there exists $m\in\mathbb{N}$ such that $x=y^m$ or $y=x^m$. The  proper  power graph $P_0 (G)$  is defined as the $1$-deleted subgraph  of $P(G).$ While $P (G)$  is obviously connected, $P_0 (G)$ may not be, and the counting of its components is an interesting topic. The reader is referred to \cite{sur} for survey about power graphs.
 In two forthcoming papers, \cite{BIS1} and \cite{BIS2}, we will apply the general method developed here to that issue, with particular attention to permutation groups. Actually, if $G$ is the symmetric or the alternating group there exists a complete orbit homomorphism which is very natural to be considered for an application of Theorem B. Those results  seem  to also have promising applications to simple and almost simple groups.

In addition to developing the tools for counting components of a graph using homomorphisms, we also compare various classes of homomorphisms (Lemma \ref{intersection}, Propositions \ref{propc} and \ref{union2cons}).

\vskip 0.6 true cm
\section{\bf Graphs}\label{graphs}
\vskip 0.4 true cm
For a finite set $A$ and $k\in \mathbb{N}$, let $\binom{A}{k}$ be the set of the subsets of $A$ of size $k.$
A {\it graph} $X=(V_X,E_X)$ is a pair of finite sets such that $V_X\neq\varnothing$ is the set of vertices, and $E_X$
 is the set of edges which is the union of the set of {\it loops} $L_X=\binom{V_X}{1}$
and a set of {\it proper edges} $E^*_X\subseteq \binom{V_X}{2}$. Note that $E^*_X$ may be empty. We usually specify the edges of a graph $X$ giving only $E^*_X$.

Let $X$ be a graph. A {\it subgraph} $\hat{X}=(V_{\hat{X}},E_{\hat{X}})$ of $X$ is a graph such that $V_{\hat{X}}\subseteq V_{X}$ and $E_{\hat{X}}\subseteq E_{X}$. If  $\hat{X}$ is a subgraph of $X$, we write $\hat{X}\subseteq X.$ For $s\in\mathbb{N}\cup\{0\}$, a subgraph $\gamma$ of $X$ such that $V_{\gamma}=\{x_i: 0\leq i\leq s\}$ with distinct $x_i\in V_X$ and $E^*_{\gamma}=\{\{x_i,x_{i+1}\} : 0\leq i\leq s-1\}$, is called a  {\it path} of length $s$ between $x_0$ and $x_s$.
Given $U\subseteq V_X$, the {\it subgraph induced} by $U$ is the subgraph $\hat{U}$ of $X$ having $V_{\hat{U}}=U$ and $E_{\hat{U}}=\{\{x_1,x_2\}\in E_X:x_1,x_2\in U\}$. A subgraph is called {\it induced} if it is the subgraph induced by some subset of vertices.
Two vertices  $x_1,x_2\in V_X$ are said to be {\it connected} in $X$ if  there exists a path between $x_1$ and $x_2.$  $X$ is called  {\it connected} if every pair of its vertices is connected.
It is well known that connectedness is an equivalence relation on $V_X.$
 Any subgraph  of $X$ induced by a connectedness equivalence class, is called a {\it component} of $X$. Equivalently, a component of $X$ is a maximal connected subgraph of $X$. It is easily checked that the vertices (the edges) of the components of $X$ give a partition of $V_X$ ($E_X$); a connected subgraph $\hat{X}$ of $X$  is a component
 if and only if $x_1\in V_{\hat{X}}$ and $\{x_1,x_2\}\in E_X$ imply $\{x_1,x_2\}\in E_{\hat{X}}.$ The component of $X$ containing $x\in V_X$ is denoted by $C_{X}(x)$.
If the only vertex of $C_{X}(x)$ is $x$, we say that $x$ (the component $C_{X}(x)$) is an {\it isolated vertex}.
The set of components of $X$ is denoted by $\mathcal{C}(X)$ and its size by $c(X).$
Given $x\in V_X$, the {\it neighborhood} of $x$ is the subset of $V_X$ defined by $N_X(x)=\{u\in V_X: \{x,u\}\in E_X\}.$ Note that $x\in N_X(x)$ by reflexivity.  

When dealing with a unique fixed  graph $X$, we usually omit the subscript $X$ in all the above notation. The terminology not  explicitly introduced is standard and can be find in \cite{dl}.

\vskip 0.6 true cm
\section{\bf  Quotient graphs and number of components}\label{general}
\vskip 0.4 true cm

Let $X=(V,E)$ be a graph and  $\sim$ be an equivalence relation on $V$. For every $x\in V,$ denote by $[x]$ the equivalence class of $x$ and call it a {\it cell}. Thus, for $x,y\in V$, we have  $[x]=[y]$ if and only if $x\sim y$ and the elements of the partition $V/\hspace{-1mm}\sim$ of $V$ associated to $\sim$ are represented by $[x]$, for $x\in V.$ The {\it quotient graph} of $X$ with respect to $\sim$, denoted by $X/\hspace{-1mm}\sim,$ is the graph with vertex set $[V]=V/\hspace{-1mm}\sim$ and edge set $[E]$  defined as follows: for every $[x_1]\in [V]$ and $[x_2]\in [V]$,  $\{[x_1],[x_2]\}\in [E]$ if  there exist $\tilde{x}_1,\tilde{x}_2\in V$ such that $\tilde{x}_1\sim x_1,\  \tilde{x}_2\sim x_2$ and $ \{\tilde{x}_1,\tilde{x}_2\}\in E.$

Passing from a  graph $X$ to a quotient graph $X/\hspace{-1mm}\sim$ reduces the complexity and obviously different equivalence relations  imply different levels of complexity reduction.  For instance, in the extreme case of the total equivalence relation, which reduces $X$ to a single vertex, all information about the graph $X$ is lost.
By an appropriate choice of the equivalence relation, we may produce a less complex quotient graph while maintaining a relationship between components of the graph and its quotient. The easiest case is when the equivalence classes are each contained in a single component, in which case  $c(X)=c(X/\hspace{-1mm}\sim).$
 \begin{defn}\label{sim}  {\rm  Let  $X=(V,E)$ be a graph and $\sim$ be an equivalence relation  on $V$. We say that
$\sim$  is {\it tame} if for every $x,\tilde{x}\in V$, $[x]=[\tilde{x}]$ implies $C_X(x)=C_X(\tilde{x}).$
We say that $X/\hspace{-1mm}\sim$ is a {\it tame} quotient of $X$  if  $\sim$  is tame.}
\end{defn}
Obviously every graph $X$ admits tame equivalence relations on its vertex set. One example is given by the relation identifying all the vertices in the same component.
 Note also that, if $X$ is connected, each equivalence relation $\sim$ on $V$ is tame.

\begin{prop}\label{quotient-graph}   Let  $X=(V,E)$ be a graph and $\sim$ be an equivalence relation  on $V$. Then:
\begin{itemize}
\item[(i)] $c(X/\hspace{-1mm}\sim)\leq c(X)$;
\item[(ii)] $c(X/\hspace{-1mm}\sim)=c(X)$ if  and only if $\sim$ is tame;
\item[(iii)]  $X$ is connected if and only if $X/\hspace{-1mm}\sim$ is connected and tame.
\end{itemize}
\end{prop}
\begin{proof} Note first that the map
$f: \mathcal{C}(X)\rightarrow \mathcal{C}(X/\hspace{-2mm}\sim)$ defined by $f(C_{X}(x))=C_{X/\sim}([x])$ for all $x\in V$
is well defined as the quotient construction respects adjacency, and hence connectedness of any pair of vertices.

(i) The map $f$ is obviously surjective, so $$c(X/\hspace{-1mm}\sim)=|\mathcal{C}(X/\hspace{-1mm}\sim)|\leq |\mathcal{C}(X)|=c(X).$$

(ii) By (i) and by the definition of $f$, $c(X/\hspace{-1mm}\sim)=c(X)$ holds if and only if $C_{X/\sim}([x])=C_{X/\sim}([y])$ implies $C_{X}(x)=C_{X}(y)$ for all $x,y\in V.$
 Suppose $c(X/\hspace{-1mm}\sim)=c(X)$ and let $[x]=[y]$, for some $x,y\in V.$ Then $C_{X/\sim}([x])=C_{X/\sim}([y])$ and therefore $C_{X}(x)=C_{X}(y)$, so $\sim$ is tame. Conversely, suppose $\sim$ is tame and let $C_{X/\sim}([x])=C_{X/\sim}([y])$ for some $x,y\in V.$ Then  in $X/\hspace{-1mm}\sim$ there is a path $\gamma$ between $[x]$ and $[y]$.
 Observe first that if $u,v\in V$ are such that $\{[u],[v]\}\in [E]$, then $u$ and $v$ are connected in $X$. Indeed, by definition of edge in a quotient graph, there exist $\tilde{u},\tilde{v}\in V$ such that $[\tilde{u}]=[u],\  [\tilde{v}]=[v]$ and $ \{\tilde{u},\tilde{v}\}\in E.$ Thus $\tilde{u}$ and $\tilde{v}$ are connected in $X$ and,  $\sim$ being tame, $u$ and $\tilde{u}$ as well as $v$ and $\tilde{v}$ are connected in $X$. By transitivity of connectedness, we then also have $u$ and $v$ connected in $X$.
Now, by an obvious inductive argument on the length of $\gamma$, we deduce that $x$ and $y$ are connected in $X$. Thus $C_{X}(x)=C_{X}(y).$

(iii)  Let $X$ be connected. Then $\sim$ is trivially tame. Moreover $c(X)=1$ so that, by (i), $c(X/\hspace{-1mm}\sim)=1$ which says that $X/\hspace{-1mm}\sim$ is connected.
Conversely, let $X/\hspace{-1mm}\sim$ be connected and tame. Then, by (ii), $c(X)=c(X/\hspace{-1mm}\sim)=1$ and so $X$ is connected.
\end{proof}

\vskip 0.6 true cm
\section{\bf Homomorphisms of graphs and partitions}\label{hom-pa}
\vskip 0.4 true cm

Let $X$ be a graph and suppose that you want to compute $c(X)$ by looking at the components of a quotient $X/\hspace{-1mm}\sim$ whose components are easier to interpret. To that end, dealing only with tame quotients is surely too restrictive. It turns out to be useful to introduce quotients which substantially reduce the complexity of $X$ at the cost of changing, in some controlled way, the number of components. To develop this idea we must  isolate a set of crucial definitions qualifying  the graph homomorphisms. Recall that the word graph always means a finite, undirected, simple and reflexive graph. Throughout the next sections,  let $X$, $Y$ be fixed graphs. We do not explicitly repeat that assumption any more.

\subsection{Maps and admissibility} Let $A$ be a set and
 $\varphi:V_X \rightarrow A$ be a map. For every $y\in A$ the subset of $V_X$ given by $\varphi^{-1}(y)$ is called the {\it fibre} of $\varphi$ on $y$. The relation $\sim_{\varphi}$
 on $V_X$ defined, for every $x,y\in V_X$, by $x\sim_{\varphi}y$ if $\varphi(x)=\varphi(y),$ is an equivalence relation. The equivalence classes of  $\sim_{\varphi}$ are called $\varphi$-{\it cells} and coincide with the nonempty fibres of $\varphi$. We call $\sim_{\varphi}$ the equivalence relation induced by $\varphi$ and denote the corresponding quotient graph  by $X/\hspace{-1mm}\sim_{\varphi}$. The above considerations allow us to transfer terminology from partitions to maps.

 Given $U\subseteq V_X$ and $y\in A$, define the {\it multiplicity} of $y$ in $U$ by the non-negative integer
\[k_{U}(y)=|U\cap \varphi^{-1}(y) |.\]
 In other words $k_{U}(y)$ is the size of the intersection between $U$ and the fibre of $\varphi$ on $y$.
We say that $y$ is  {\it admissible} for $U$ (or $U$ is admissible for $y$), if $k_{U}(y)>0$. Note that $y$ is  admissible for $U$ if and only if $y\in \varphi(U).$ Thus $\varphi(U)$ is the subset of elements of $A$ admissible for $U.$
If $\hat{X}$ is a subgraph of $X$ we adopt the same language referring to $V_{\hat{X}}$ and we define $k_{\hat{X}}(y)$ by  $k_{V_{\hat{X}}}(y).$ In the sequel, the concepts of admissibility and of multiplicity reveal themselves very useful when the subgraph under consideration is a component of $X$.
Note that $k_{X}(y)$ is simply the size of the fibre  $\varphi^{-1}(y)$.
We will usually apply the above ideas when $A$ is the vertex set of some graph.

\subsection{Homomorphisms}
 Let  $\varphi:V_X \rightarrow V_Y$ be  a map. Then $\varphi$ is called
a {\it homomorphism}  from $X$ to $Y$  if, for each $x_1,x_2\in V_X,$ $\{x_1,x_2\}\in E_X$ implies $\{\varphi(x_1),\varphi(x_2)\}\in E_Y$. The set of the homomorphisms  from $X$ to $Y$ is denoted by $\mathrm{Hom}(X,Y).$ $\varphi\in\mathrm{Hom}(X,Y)$ is called  {\it surjective} ({\it injective, bijective}) if $\varphi:V_X\rightarrow V_Y$ is surjective (injective, bijective). We denote the set of surjective homomorphisms from $X$ to $Y$  by $\mathrm{Sur}(X,Y).$

Note that  a map $\varphi:V_X \rightarrow V_Y$ is a homomorphism from $X$ to $Y$ if and only if
\begin{equation}\label{hom-neigh}
\forall x\in V_X,\  \varphi(N_X(x))\subseteq N_Y(\varphi(x)).
 \end{equation}

Let  $\varphi\in\mathrm{Hom}(X,Y).$ Observe that $\varphi$ may map a proper edge of $X$ to a loop of $Y.$
Moreover, $\varphi$ induces a map between $E_X$ and $E_Y,$ associating to  every edge $e=\{x_1,x_2\}\in E_X$, the edge $\varphi(e)=\{\varphi(x_1),\varphi(x_2)\}\in E_Y$. We denote that map  between $E_X$ and $E_Y$  again with $\varphi.$ We also use the notation $\varphi:X \rightarrow Y$ to indicate the homomorphism $\varphi.$

An important example of surjective homomorphism is given by the projection on the quotient. Consider a quotient graph $X/\hspace{-1mm}\sim$ and let $\pi:V_X\rightarrow [V_X]$ be the map defined by $\pi(x)=[x],$ for all $x\in V_X$. If $\{x_1,x_2\}\in E_X$, then  we surely have $\{[x_1],[x_2]\}\in [E_X]$. Thus $\pi\in \mathrm{Sur}(X,X/\hspace{-1mm}\sim)$  and $\pi$ is called the {\it projection} on the quotient graph.

If $\hat{X}$ is a subgraph of $ X$, then  the {\it image} of $\hat{X}$ by $\varphi\in\mathrm{Hom}(X,Y)$ is defined as the subgraph of $Y$ given by $\varphi(\hat{X})=(\varphi(V_{\hat{X}}), \varphi(E_{\hat{X}}))$. Observe that, generally, if  $\hat{X}\subseteq X$ then  $\varphi(\hat{X})$ is not an induced subgraph of $Y$. In particular, the condition $\varphi\in \mathrm{Sur}(X,Y)$  is weaker than $\varphi(X)=Y,$ because the surjectivity requires only $\varphi(V_X)=V_Y$  while $\varphi(X)=Y$ requires both $\varphi(V_X)=V_Y$ and  $\varphi(E_X)=E_Y.$

\begin{defn}\label{comp} {\rm Let $\varphi\in \mathrm{Hom}(X,Y)$. Then $\varphi$ is called:
\begin{itemize}
\item[(a)]  {\it complete} if $\varphi(X)=Y$. We denote the set of  complete homomorphisms from $X$ to $Y$  by $\mathrm{Com}(X,Y)$;
\item[(b)] an {\it isomomorphism} if $\varphi$ is bijective and complete. We denote the set of  isomorphisms from $X$ to $Y$  by $\mathrm{Iso}(X,Y)$. If  $ \mathrm{Iso}(X,Y)\neq \varnothing,$ we say that $X$ and $Y$ are isomorphic and  we write $X\cong Y$;
\item[(c)] {\it tame} if $\sim_{\varphi}$ is tame.  We denote the set of tame homomorphisms from $X$ to $Y$ by $\mathrm{T}(X,Y).$
\end{itemize}}
\end{defn}
We make a few comments on these definitions. First of all note that $\varphi$ is tame if and only if every fibre of $\varphi$ is connected.  Note also that
 the composition of complete homomorphisms is a complete homomorphism and that
\begin{equation}\label{sub}
\mathrm{Iso}(X,Y)\subseteq \mathrm{Com}(X,Y)\subseteq \mathrm{Sur}(X,Y)\subseteq \mathrm{Hom}(X,Y).
\end{equation}
Finally note that,  each homomorphism $\varphi\in \mathrm{Hom}(X,Y)$ induces a complete homomorphism from $X$ to  $\varphi(X)$. Our strong interest in  completeness is motivated by the fact that the projection on the quotient graph  is a complete surjective homomorphism.

\begin{lem}\label{nat-comp} Let $X$ be a graph and $\sim$ an equivalence relation on $V_X$. Then $\pi\in\mathrm{Com}(X,X/\hspace{-1mm}\sim).$
\end{lem}
\begin{proof}  Since $\pi\in \mathrm{Sur}(X,X/\hspace{-1mm}\sim),$ we need only check that $[E_X]\subseteq \pi(E_X).$ Pick  $e=\{[x_1],[x_2]\}\in [E_X]$, with $x_1,x_2\in V_X$. Then, by definition of quotient graph, there exist $\tilde{x}_1,\tilde{x}_2\in V_X$ such that $\tilde{x}_1\sim x_1,\  \tilde{x}_2\sim x_2$ and $ \{\tilde{x}_1,\tilde{x}_2\}\in E_X.$ Thus, we have $\pi(\{\tilde{x}_1,\tilde{x}_2\})=\{\pi(\tilde{x}_1),\pi(\tilde{x}_2)\}=e.$
\end{proof}

The following lemma shows that $X/\hspace{-1mm}\sim_{\varphi}$ is isomorphic to $Y$ when $\varphi\in \mathrm{Com}(X,Y)$ and enables us to interpret every quotient graph of $X$  as the image of $X$ under a complete homomorphism.

 \begin{lem}\label{h}  Let $\varphi\in  \mathrm{Hom}(X,Y)$ and let $$\tilde{\varphi}:X/\hspace{-1mm}\sim_{\varphi}\rightarrow Y$$ be the map defined by $\tilde{\varphi}([x])=\varphi(x)$ for all $[x]\in [V_X].$ Then:
 \begin{itemize}
 \item[(i)] $\tilde{\varphi}$ is an injective homomorphism, and $\tilde{\varphi}$ is surjective if and only if $\varphi$ is surjective;
 \item[(ii)] $\tilde{\varphi}$ is an isomorphism if and only if $\varphi$ is complete.
 \end{itemize}
  \end{lem}
\begin{proof} (i) This is just  \cite[Theorem 1.6.10] {kna}.

(ii) Suppose $\varphi$ is complete. Thus $\varphi$ is also surjective and, by (i), $\tilde{\varphi}$ is a bijective homomorphism. On the one hand, due to $\tilde{\varphi}([E_X])=\varphi(E_X)=E_Y,$  $\tilde{\varphi}$ is also complete and hence an isomorphism.
Assume now that $\tilde{\varphi}$ is an isomorphism. By definition of $\tilde{\varphi}$, we have $\varphi=\tilde{\varphi}\circ \pi$. On the other hand, by Lemma \ref{nat-comp}, $\pi$ is complete and by \eqref{sub}, also $\tilde{\varphi}$  is complete. Thus $\varphi$ is complete because it is a composition of complete homomorphisms.
\end{proof}

\subsection{Equitable and orbit partitions}\label{equi-orb}

We recall some classic types of partitions and extend the definitions to the context of homomorphisms.

\begin{defn}\label{def-eq} {\rm Let $\mathcal{P}=\{P_1,\dots, P_k\}$ be a partition of $V_X$. Then $\mathcal{P}$ is called:
\begin{itemize}
\item[(a)]an {\it equitable partition} if, for every $i,j\in \{1,\dots, k\}$, the size of $N_X(x)\cap P_j$ is the same for all  $x\in P_i$.
 We call $\varphi\in  \mathrm{Hom}(X,Y)$ an {\it equitable homomorphism} if the partition into $\varphi$-cells is equitable. The set of equitable homomorphisms is denoted by $ \mathrm{E}(X,Y)$;
\item[(b)]an  {\it orbit partition} if $\mathcal{P}$ is the set of orbits of some $\mathfrak{G}\leq  \mathrm{Aut}(X)$. We call $\varphi\in  \mathrm{Hom}(X,Y)$ an  {\it orbit homomorphism} (with respect to $\mathfrak{G}$) if the partition into $\varphi$-cells is an orbit partition (with respect to $\mathfrak{G}$).
The set of orbit homomorphisms is denoted by $ \mathrm{O}(X,Y).$ If $\varphi\in  \mathrm{O}(X,Y)$ is an orbit homomorphism with respect to $\mathfrak{G}$ we say briefly that $\varphi$ is $\mathfrak{G}$-{\it consistent} or that $\mathfrak{G}$ is $\varphi$-{\it consistent}. \end{itemize}}
\end{defn}

It is well known that any orbit partition is an equitable partition  but the converse does not hold (\cite[Proposition 9.3.5]{Ler}).  Thus we have $ \mathrm{O}(X,Y)\subseteq \mathrm{E}(X,Y)$ with a  proper inclusion in general. Since the partition with each cell containing just a vertex is the orbit partition relative to the identity subgroup of  $\mathrm{Aut}(X)$, we also have $\mathrm{Iso}(X,Y)\subseteq \mathrm{O}(X,Y)\cap \mathrm{Com}(X,Y).$
Once the graph $X$ is fixed, the homomorphisms $\varphi\in \mathrm{O}(X,Y)\cap\mathrm{Com}(X,Y),$ for some graph $Y,$
can be easily described in terms of graph automorphisms of $X$.
Indeed, pick $\mathfrak{G}\leq  \mathrm{Aut}(X)$ and let $\sim_\mathfrak{G}$ be the corresponding orbit partition of $V_X.$ Then the projection  onto the quotient graph $Y=X/\hspace{-1mm}\sim_\mathfrak{G}$ belongs to $\mathrm{O}(X,Y)\cap\mathrm{Com}(X,Y)$. Conversely let $\varphi\in \mathrm{O}(X,Y)\cap\mathrm{Com}(X,Y)$ be an orbit homomorphism with respect to $\mathfrak{G}\leq \mathrm{Aut}(X)$. Then, by Lemma \ref{h}, $\varphi$ coincides up to an isomorphism with the projection on $X/\hspace{-1mm}\sim_{\varphi}=X/\hspace{-1mm}\sim_\mathfrak{G}.$

The following equivalent formulation for the $\varphi$-consistency is immediate.
\begin{lem}\label{phi-con}  Let  $\varphi\in\mathrm{Hom}(X,Y).$ A group $\mathfrak{G}\leq {\mathrm {Aut}}(\Gamma)$ is  $\varphi$-consistent if and only if the following two conditions are satisfied:
 \begin{itemize}
\item[(a)]  $\varphi\circ f=\varphi,\  \forall f\in \mathfrak{G};$
 \item[(b)]  for each $x_1,x_2\in V_X$ with $\varphi(x_1)=\varphi(x_2)$, there exists $f\in \mathfrak{G}$ such that $x_2=f(x_1).$
 \end{itemize}
\end{lem}

\vskip 0.6 true cm
\section{\bf  Homomorphisms and  components}
\vskip 0.4 true cm

Given a generic $\varphi\in \mathrm{Hom}(X,Y)$,  the relation between the components in the graphs $X$ and $Y$ is quite poor. Obviously, the following fact holds.

\begin{lem}\label{quasi-path} Let $\varphi\in \mathrm{Hom}(X,Y)$.
If $\hat{X}$  is a connected subgraph of $X$  then $\varphi(\hat{X})$ is connected.
\end{lem}

Thus, if $C\in\mathcal{C}(X)$, then $\varphi(C)$ is a connected subgraph of $Y$ but it is not necessarily a component. The best we can say is that there exists a unique component $C'\in\mathcal{C}(Y)$ such that $\varphi(C)\subseteq C'.$ Unfortunately things do not improve if $\varphi\in \mathrm{Com}(X,Y).$
Consider as a very basic example, the graph $X$ with $$V_X=\{1a,1b, 2,3\}, \quad E^*_X=\{\{1a,3\}, \{1b,2\}\}$$ and the equivalence relation $\sim$ on $V_X$ defined only by $1a\sim 1b$. Then $Y=X/\hspace{-1mm}\sim$ is connected and is a path of length $2.$ Now look at the complete homomorphism $\pi:X\rightarrow Y$ given by the natural projection. $\pi$ takes the component $C$ of $X$ having $V_C=\{1a,3\}$ into the connected subgraph $\pi(C)$  such that $V_{\pi(C)}=\{[1a],[3]\}$ and $E^*_{\pi(C)}= \{\{[1a],[3]\}\}$. Thus $\pi(C)$, being  a path of length $1,$ is different from the only component of $Y$.
Nevertheless there is a specific situation which is worth discussing.
\begin{prop}\label{main-component} Let $\varphi\in \mathrm{Com}(X,Y)$ and assume that every component of $X$ apart from a unique $C\in \mathcal{C}(X)$ is an isolated vertex. Let $C'\in \mathcal{C}(Y)$ be the only component of $Y$ such that $\varphi(C)\subseteq C'$. If $V_{C'}=V_{\varphi(C)},$ then $\varphi(C)=C'.$
\end{prop}
\begin{proof}  We know that
 $C'$ and $\varphi(C)$ have the same vertices so  that we just need to show that they also have the same edges. Since a component is always an induced subgraph, we trivially have $E_{\varphi(C)}\subseteq E_{C'}.$ To show the other inclusion it is enough to show that $E^*_{C'}\subseteq E_{\varphi(C)}.$
Let $e'=\{y_1,y_2\}\in E^*_{C'}$, for some distinct $y_1,y_2\in V_{C'}$. Then, by the completeness of  $\varphi$, there exist $x_1,x_2\in V_X$  such that  $\varphi(x_1)=y_1,\varphi(x_2)=y_2 $ and $e=\{x_1,x_2\}\in E_X.$ As $y_1\neq y_2$ we also have $x_1\neq x_2$. Thus  $e\in E^*_X$, which implies that $x_1$ and $x_2$ are not isolated in $X$. But  if a component of $X$ is not an isolated vertex, it coincides with $C$. It follows that $x_1,x_2\in V_{C}$ and so $e\in E_{C}.$ Hence $e'=\varphi(e)\in E_{\varphi(C)}.$
\end{proof}

We now consider some well known types of homomorphisms.
By \eqref{hom-neigh}, every graph homomorphism $\varphi\in \mathrm{Hom}(X,Y)$ maps $N_X(x)$ into $N_Y(\varphi(x))$  for all  $x\in V_X.$  Denoting by $\varphi_{\mid N_X(x)}:N_X(x)\rightarrow N_Y(\varphi(x))$ the corresponding restriction homomorphism, the {\it locally constrained graph homomorphisms} are those requiring an additional condition on the map $\varphi_{\mid N_X(x)}$  for all  $x\in V_X.$
\begin{defn}\label{classic}
 {\rm Let $\varphi\in \mathrm{Hom}(X,Y)$. Then $\varphi$ is called {\it locally surjective (injective, bijective)} if, for every $x\in V_X$, $\varphi_{\mid N_X(x)}$ is surjective (injective, bijective). We denote the set of the locally surjective (injective, bijective) homomorphisms  by $\mathrm{LSur}(X,Y)$ (by $\mathrm{LIn}(X,Y)$, $\mathrm{LIso}(X,Y)$).}
\end{defn}
 An exhaustive survey of the three types of locally constrained graph homomorphisms  defined above  is given in \cite{FIKRA} to which we refer the reader for a wide overview on the many applications in different areas, from graph theory and combinatorial topology to computer science and social behaviour.
We will be particularly interested in the locally surjective homomorphisms because they represent
a manageable and wide class of homomorphisms which guarantee the natural migration of the components (see Proposition \ref{hom-con}).
Note that, by \eqref{hom-neigh}, $\varphi\in \mathrm{Hom}(X,Y)$ is locally surjective if and only if
\begin{equation}\label{hom-neigh-ls}
\forall x\in V_X,\    N_Y(\varphi(x))\subseteq \varphi(N_X(x)).
 \end{equation}
Note also that being locally surjective does not imply being surjective.

We next recall the  class of locally strong homomorphisms which, appearint for the first time in \cite{Pul}, were later used in the study of the endomorphism spectrum of a graph \cite{kna-paper}.
\begin{defn}\label{locstr}
 {\rm Let $\varphi\in \mathrm{Hom}(X,Y)$. Then $\varphi$ is called {\it locally strong}  if,  for every $x_1,x_2\in V_X$, $\{\varphi(x_1),\varphi(x_2)\}\in E_Y$ implies that, for every  $\tilde{x}_1\in \varphi^{-1}(\varphi(x_1))$,
 there exists $\tilde{x}_2\in \varphi^{-1}(\varphi(x_2))$ such that  $\{\tilde{x}_1,\tilde{x}_2\}\in E_X.$ We denote the set of the locally strong homomorphisms  by $\mathrm{LS}(X,Y).$}
\end{defn}

We show that being locally surjective implies being locally strong and that these two classes coincide in the context of surjective homomorphisms. To this end, we first present a useful characterisation of the locally strong homomorphisms.
\begin{lem}\label{formulation} $\varphi\in \mathrm{LS}(X,Y)$ if and only if, for every $x_1,x_2\in V_X$, $\{\varphi(x_1),\varphi(x_2)\}\in E_Y$ implies that there exists $\tilde{x}_2\in \varphi^{-1}(\varphi(x_2))$ such that $\{x_1,\tilde{x}_2\}\in E_X.$
\end{lem}
\begin{proof} Let $\varphi\in \mathrm{LS}(X,Y)$ and let $x_1,x_2\in V_X$ be such that $\{\varphi(x_1),\varphi(x_2)\}\in E_Y$. Since $x_1\in \varphi^{-1}(\varphi(x_1))$, $\varphi$ locally strong implies that there exists $\tilde{x}_2\in \varphi^{-1}(\varphi(x_2))$  with  $\{x_1,\tilde{x}_2\}\in E_X.$

Assume next that, for every $x_1,x_2\in V_X$, $\{\varphi(x_1),\varphi(x_2)\}\in E_Y$ implies that there exists $\tilde{x}_2\in \varphi^{-1}(\varphi(x_2))$ such that $\{x_1,\tilde{x}_2\}\in E_X.$ We show that $\varphi\in \mathrm{LS}(X,Y)$. Let $x_1,x_2\in V_X$ be such that $e=\{\varphi(x_1),\varphi(x_2)\}\in E_Y$ and pick any $\tilde{x}_1\in \varphi^{-1}(\varphi(x_1))$. Then $e=\{\varphi(\tilde{x}_1),\varphi(x_2)\}$ and so, applying the assumption to $\tilde{x}_1, x_2$, we obtain the existence of $\tilde{x}_2\in \varphi^{-1}(\varphi(x_2))$ such that  $\{\tilde{x}_1,\tilde{x}_2\}\in E_X.$
\end{proof}

\begin{lem}\label{intersection}  Let $X$ and $Y$ be graphs. Then the following hold:
\begin{itemize}\item[(i)] $\mathrm{LSur}(X,Y)\subseteq \mathrm{LS}(X,Y)$;
\item[(ii)]$\mathrm{LS}(X,Y)\cap \mathrm{Sur}(X,Y)=\mathrm{LSur}(X,Y)\cap \mathrm{Sur}(X,Y).$
\end{itemize}
\end{lem}
\begin{proof}(i) Let $\varphi\in\mathrm{LSur}(X,Y)$. By Lemma \ref{formulation}, we need to show that for every $x_1,x_2\in V_X$, $\{\varphi(x_1),\varphi(x_2)\}\in E_Y$ implies that there exists $\tilde{x}_2\in \varphi^{-1}(\varphi(x_2))$ such that $\{x_1,\tilde{x}_2\}\in E_X.$ Indeed, if $\{\varphi(x_1),\varphi(x_2)\}\in E_Y$, we have that $\varphi(x_2)\in N_Y(\varphi(x_1))$ and, since $\varphi\in\mathrm{LSur}(X,Y)$, we have that $N_Y(\varphi(x_1))=\varphi(N_X(x_1)).$ Hence there exists $\tilde{x}_2\in N_X(x_1)$ such that $\varphi(\tilde{x}_2)=\varphi(x_2)$, which means $\{x_1,\tilde{x}_2\}\in E_X$ and $\tilde{x}_2\in \varphi^{-1}(\varphi(x_2))$.

(ii) By (i), it is enough to show that $\mathrm{LS}(X,Y)\cap \mathrm{Sur}(X,Y)\subseteq\mathrm{LSur}(X,Y)\cap \mathrm{Sur}(X,Y).$
Let then $\varphi\in \mathrm{LS}(X,Y)\cap \mathrm{Sur}(X,Y)$ and show that $\varphi\in\mathrm{LSur}(X,Y).$ We need to see that, for every $x\in V_X$, $\varphi(N_X(x))=N_Y(\varphi(x)).$ One inclusion is obvious by \eqref{hom-neigh} and therefore we need only to show that $N_Y(\varphi(x))\subseteq \varphi(N_X(x)).$ Let $y\in N_Y(\varphi(x))$. Then $\{y,\varphi(x)\}\in E_Y$ and,  $\varphi$ being surjective, there exists $x'\in V_X$ such that $y=\varphi(x').$ Thus, as $\{\varphi(x), \varphi(x')\}\in E_Y$ and $\varphi$ is locally strong, there exists $\tilde{x}'\in V_X$ such that $\{x,\tilde{x}'\}\in E_X$ and $\varphi(\tilde{x}')=\varphi(x')=y.$ Hence $y\in \varphi(N_X(x)).$
\end{proof}
Generally, $\mathrm{LSur}(X,Y)\subsetneq\mathrm{LS}(X,Y).$ Consider, for instance, the graph $X$ with $V_X=\{1\}$ and $E^*_X=\varnothing$; the graph $Y$ with $V_Y=\{a,b\}$ and $E^*_Y=\{\{a,b\}\}$; $\varphi\in\mathrm{Hom}(X,Y)$ defined by $\varphi(1)=a$. Then, trivially, $\varphi\in \mathrm{LS}(X,Y)$ but $\varphi\notin \mathrm{LSur}(X,Y).$

\begin{defn}\label{pseudo}
 {\rm $\varphi\in \mathrm{Hom}(X,Y)$ is called {\it pseudo-covering} if $\varphi\in \mathrm{LS}(X,Y)\cap \mathrm{Sur}(X,Y)$. We denote the set of the pseudo-covering homomorphisms from $X$ to $Y$  by $\mathrm{PC}(X,Y).$}
\end{defn}

Observe that for a projection on a quotient graph, being pseudo-covering is equivalent to being locally strong as well as to being locally surjective.
Lemma \ref{intersection} makes clear two good reasons to adopt the term pseudo-covering. First of all in  \cite[Definition 1.7]{pr} a graph is called a pseudo-cover of its quotient graph when the natural projection is  locally strong. Secondly the word covering is typically used in the context of locally constrained graph homomorphisms. More precisely, if $\varphi\in \mathrm{LIso}(X,Y)\cap  \mathrm{Sur}(X,Y)$, then $\varphi$ is called a covering (\cite[Section 6.8]{god}); if $\varphi\in \mathrm{LIn}(X,Y)$, then $\varphi$ is called a partial covering (\cite{FI2}). So, in some sense, we are filling a vacancy of terminology, with respect to the concept of covering, in the locally surjective case. Note also that pseudo-covering homomorphisms are considered in \cite{FI} with the name of global role assignments. There it is proved that the problem  of deciding if, given a graph $Y$, we have $\mathrm{PC}(X,Y)\neq\varnothing$,  for some input graph $X$, is $\mathrm{NP}$-complete, with the exception of the case in which all the components of $Y$ have at most two vertices.

\begin{lem}\label{h1}  Let $\varphi\in  \mathrm{Com}(X,Y)$ and let $\pi$ the projection of $X$ onto $X/\hspace{-1mm}\sim_{\varphi}$. Then $\pi$ is pseudo-covering (locally surjective, locally injective, locally bijective, locally strong) if and only if $\varphi$ is.
 \end{lem}
\begin{proof} Using the notation of Lemma \ref{h} we have $\tilde{\varphi}\circ\pi=\varphi$ and, since $\varphi\in \mathrm{Com}(X,Y)$,  $\tilde{\varphi}$ is an isomorphism. Since the composition of a pseudo-covering (locally surjective, locally injective, locally bijective,  locally strong) homomorphism with an isomorphism is pseudo-covering (locally surjective, locally injective, locally bijective,  locally strong), the assertion follows.
 \end{proof}

\begin{prop}\label{propc}  Let $X$, $Y$ and $Z$ be graphs.
\begin{itemize}
\item[(i)] If $\varphi\in \mathrm{PC}(X,Y)$ and $\psi\in \mathrm{PC}(Y,Z),$ then $\psi\circ \varphi\in \mathrm{PC}(X,Z);$
\item[(ii)] \begin{equation}\label{sub2}
\begin{array}{l}
\mathrm{Iso}(X,Y)\subseteq \mathrm{O}(X,Y)\cap \mathrm{Com}(X,Y)\subseteq \mathrm{E}(X,Y)\cap \mathrm{Com}(X,Y)\subseteq \\
 \mathrm{LS}(X,Y)\cap \mathrm{Com}(X,Y)= \mathrm{PC}(X,Y)=\mathrm{LSur}(X,Y)\cap \mathrm{Com}(X,Y)\subseteq \mathrm{Com}(X,Y).
\end{array}
\end{equation}
\end{itemize}
\end{prop}
\begin{proof}  (i) Straightforward.

(ii) The first two inclusions follow from the discussion in Section \ref{equi-orb}. We show that $\mathrm{E}(X,Y)\cap \mathrm{Com}(X,Y)\subseteq  \mathrm{LS}(X,Y)\cap \mathrm{Com}(X,Y).$ Let $\varphi\in \mathrm{E}(X,Y)\cap \mathrm{Com}(X,Y)$  and show that $\varphi\in  \mathrm{LS}(X,Y).$ By Lemma \ref{h1}, it is enough to show that the natural projection $\pi:X\rightarrow X/\hspace{-1mm}\sim_{\varphi}$ is locally strong. By Lemma \ref{formulation}, we need to see that for every $x_1,x_2\in V_X$, $\{[x_1],[x_2]\}\in E_{ X/\sim_{\varphi}}$ implies that there exists $\tilde{x}_2\in \varphi^{-1}(\varphi(x_2))$ such that $\{x_1,\tilde{x}_2\}\in E_X.$ Now, $\{[x_1],[x_2]\}\in E_{ X/\sim_{\varphi}}$ means that there exist $x'_1, x'_2\in V_X$ such that $\{x'_1,x'_2\}\in E_X$, $\varphi(x'_1)=\varphi(x_1)$ and $\varphi(x'_2)=\varphi(x_2)$. Thus $x'_2\in N_X(x'_1)\cap \varphi^{-1}(\varphi(x_2))$ and $x_1,x'_1$ belong to the same $\varphi$-cell. Since the partition into $\varphi$-cells is equitable, we then have $N_X(x_1)\cap \varphi^{-1}(\varphi(x_2))\neq\varnothing$ and, to conclude, it suffices to pick any $\tilde{x}_2\in N_X(x_1)\cap \varphi^{-1}(\varphi(x_2))$.

Next we see that $\mathrm{PC}(X,Y)= \mathrm{LS}(X,Y)\cap \mathrm{Com}(X,Y)$. By definition of $\mathrm{PC}(X,Y)$ we have $\mathrm{LS}(X,Y)\cap\supseteq\mathrm{PC}(X,Y)\supseteq \mathrm{LS}(X,Y)\cap \mathrm{Com}(X,Y)$. Moreover as an obvious consequence of Lemma \ref{formulation}, we have $\mathrm{PC}(X,Y)\subseteq \mathrm{Com}(X,Y)$ and so $\mathrm{PC}(X,Y)\subseteq \mathrm{LS}(X,Y)\cap\mathrm{Com}(X,Y).$

The fact that $\mathrm{PC}(X,Y)=\mathrm{LSur}(X,Y)\cap \mathrm{Com}(X,Y)$ is now a consequence of Lemma \ref{intersection}.
\end{proof}

 \begin{defn} \label{phicon2} {\rm  Let $\varphi\in \mathrm{Hom}(X,Y)$.
 For $C'\in \mathcal{C}(Y)$, put
\[  \mathcal{C}(X) _{C'}=\{C\in  \mathcal{C}(X)\ :\  \varphi(C)\subseteq C' \},\qquad c (X) _{C'}=|\mathcal{C}(X) _{C'}|. \]}
\end{defn}

\begin{prop}\label{hom-con} Let $\varphi\in \mathrm{LSur}(X,Y)$.
 \begin{itemize}
\item[(i)]  If $C\in \mathcal{C}(X)$, then $\varphi(C)\in \mathcal{C}(Y)$. In particular, the image of $X$ is a union of components of $Y$.
\item[(ii)] For every $x\in V_X$, $\varphi(C_{X}(x))=C_{Y}(\varphi(x))\cong C_{X}(x)/\hspace{-1mm}\sim_{\varphi}.$
\item[(iii)] For every $C\in\mathcal{C}(X)$,  $\varphi^{-1}(\varphi(V_C))=\displaystyle{\bigcup_{\hat{C}\in  \mathcal{C}(X)_{\varphi(C)}}V_{\hat{C}}}$.

\end{itemize}
 \end{prop}

 \begin{proof} (i) We first consider the case $\varphi\in \mathrm{PC}(X,Y)$.
 Let $C\in \mathcal{C}(X)$: we shall show that $\varphi(C)\in \mathcal{C}(Y).$   By Lemma \ref{quasi-path}, $\varphi(C)$ is a  connected subgraph and we need to see  that it is maximal connected. Assume the contrary. Then there exists an edge   $\{y,y'\}\in E_Y\setminus \varphi(E_C),$ with $y\in V_{\varphi(C)}=\varphi(V_C)$ and $y'\in V_Y$.   Let $x\in V_C$ with $y=\varphi(x).$  $\varphi$ being  surjective and locally strong, there also exists $x'\in V_X$ such that $\varphi(x')=y'$ and $\{x,x'\}\in E_X.$
Since $x\in V_C$, with $C$ a component, we then get $x'\in V_C$ and  so $e=\{x,x'\}\in E_C.$ Thus $\varphi(e)= \{\varphi(x),\varphi(x')\}\in \varphi(E_C),$ that is, $\{y,y'\}\in \varphi(E_C),$ a contradiction.

We now  consider the case $\varphi\in \mathrm{LSur}(X,Y)$. Let  $C\in \mathcal{C}(X)$ and $C'\in \mathcal{C}(Y)$ be the unique component of $Y$ containing $\varphi(C).$ Then it is easily checked that $\varphi_{\mid C}\in \mathrm{LSur}(C,C')$. By \cite[Observation 2.4]{FI},  $C'$ being connected, we also have that $\varphi_{\mid C}\in \mathrm{Sur}(C,C')$ and thus $\varphi_{\mid C}\in \mathrm{PC}(C,C')$. Since the result has been proved for pseudo-covering homomorphisms and $C$ is connected, we deduce that $\varphi(C)=C'$.

 (ii) Let $x\in V_X$. By (i), $\varphi(C_{X}(x))$ is a component of $Y$ that contains the vertex $\varphi(x)$ and thus $\varphi(C_{X}(x))=C_{Y}(\varphi(x)).$
Next observe that $\varphi$ restricted to the subgraph $C_{X}(x)$ defines a complete homomorphism onto $C_{Y}(\varphi(x)) $ and apply Lemma \ref{h}.

(iii) The fact that if $\hat{C}\in \mathcal{C}(X)_{\varphi(C)}$ then $V_{\hat{C}}\subseteq \varphi^{-1}(\varphi(V_C))$ is obvious.
Let $x\in \varphi^{-1}(\varphi(V_C))$, for some $C\in\mathcal{C}(X)$. To conclude it is enough to show that  $\varphi(C_{X}(x))=\varphi(C).$ From $\varphi(x)\in
\varphi(V_C)$, it follows that there exists $\overline{x}\in V_C$ with $\varphi(x)=\varphi(\overline{x})$. Thus, by (ii), we get \[\varphi(C_{X}(x))=C_{Y}(\varphi(x))=C_{Y}(\varphi(\overline{x}))=\varphi(C_{X}(\overline{x}))=\varphi(C).\]
 \end{proof}

As an interesting consequence, we have a comparison  between the isolated vertices of $X$ and those in $Y$ and a general link between the components of $X$ and $Y$ in the tame case.

\begin{cor}\label{isolated} Let $\varphi\in \mathrm{LSur}(X,Y)$ and $x\in V.$ If $x$ is isolated in $X$, then $\varphi(x)$ is isolated in $Y$.
  \end{cor}
 \begin{proof}  If $V_{C_{X}(x)}=\{x\}$ then, by Proposition \ref{hom-con} (ii), $V_{C_{Y}(\varphi(x))}=\{\varphi(x)\}.$
 \end{proof}

 \begin{cor}\label{tame-component} Let $\varphi\in \mathrm{PC}(X,Y)\cap  \mathrm{T}(X,Y).$ Then $\varphi$ induces a bijection between $\mathcal{C}(X)$ and $\mathcal{C}(Y)$. Given $C'\in \mathcal{C}(Y)$, if $C$ is the unique component  of $X$ such that $\varphi(C)=C'$, then $V_C=\varphi^{-1}(V_{C'}).$
  \end{cor}
 \begin{proof} By Proposition \ref{hom-con}, we can define the map $\varphi_{\mathcal{C}}:\mathcal{C}(X)\rightarrow \mathcal{C}(Y)$ by $\varphi_{\mathcal{C}}(C)=\varphi(C)$ for all $C\in \mathcal{C}(X)$.  $\varphi$ being surjective, $\varphi_{\mathcal{C}}$ is surjective too. Since $\sim_{\varphi}$ is tame, Proposition \ref{quotient-graph}, gives $c(X)=c(X/\hspace{-1mm}\sim)$. On the other hand, $\varphi$  being complete, Lemma \ref{h}, guarantees that $Y\cong X/\hspace{-1mm}\sim$ and thus $c(Y)=c(X/\hspace{-1mm}\sim)$, so that $c(X)=c(Y).$ It follows that $\varphi_{\mathcal{C}}$ is injective.

 Let next $C'\in \mathcal{C}(Y)$ and $C\in \mathcal{C}(X)$ be the unique component such that $\varphi(C)=C'$. Surely we have $V_C\subseteq \varphi^{-1}(V_{C'}).$ To get the other inclusion let $x_1\in \varphi^{-1}(V_{C'})$ and choose $x_2\in V_C$. Since both $\varphi(x_1)$ and $\varphi(x_2)$ belong to $V_{C'},$ we have $C'=C_Y(\varphi(x_1))=C_Y(\varphi(x_2))$. Hence, by Proposition \ref{hom-con}, we have  $\varphi(C_X(x_1))=\varphi(C_X(x_2))$, that is, $\varphi_{\mathcal{C}}(C_X(x_1))=\varphi_{\mathcal{C}}(C_X(x_2)).$ Since $\varphi_{\mathcal{C}}$ is a bijection, we then get $C_X(x_1)=C_X(x_2)=C$ so that  $x_1\in V_C.$
\end{proof}

 \begin{defn}\label{quot-pse}{\rm Let $X$ be a graph and $\sim$ an equivalence relation on $V_X$.  We say that the quotient graph $X/\hspace{-1mm}\sim$ is {\it pseudo-covered}  by $X$ (is an {\it orbit quotient} of $X$), with respect to $\sim,$ if the projection $\pi:X \rightarrow X/\hspace{-1mm}\sim$ is pseudo-covering (is an orbit homomorphism). }
 \end{defn}
 Note that $X/\hspace{-1mm}\sim$ is pseudo-covered if and only if, for each $x_1,x_2\in V_X$ such that $\{[x_1],[x_2]\}\in [E_X]$ there exists $\tilde{x}_2\in V_X$ with $\{x_1,\tilde{x}_2\}\in E_X$ and $[\tilde{x}_2] =[x_2]$. We establish next a useful criterium of connectedness for $X$ relying on that of $X/\hspace{-1mm}\sim$.

\begin{cor}\label{connection}  Assume that $X/\hspace{-1mm}\sim$  is connected and pseudo-covered. If there exists $[x]\in [V_X]$ such that $[x]\subseteq V_{C_{X}(x)}$, then $X$ is connected.
  \end{cor}
  \begin{proof}  By Lemma \ref{nat-comp}, we can apply Proposition \ref{hom-con} to $\pi:X \rightarrow X/\hspace{-1mm}\sim$ obtaining that, for each $C\in \mathcal{C}(X),$  $\pi(C)=X/\hspace{-1mm}\sim.$  In particular $\pi(V_C)=[V_X]$ and thus each component contains at least one vertex in  each equivalence class with respect to $\sim.$ Since $[x]\subseteq V_{C_{X}(x)}$,  we therefore have a common vertex for $C$ and $C_{X}(x)$. Thus $C_{X}(x)=C$ is the only component in $X$.
   \end{proof}

 \vskip 0.6 true cm
 \section{\bf Counting the components }\label{new}
\vskip 0.4 true cm

Our goal is to count components of a graph $X$ by counting those of a less complex homomorphic image $Y$. We begin with a rough link between the two.

\begin{defn} \label{phicon} {\rm  Let  $\varphi\in \mathrm{Hom}(X,Y)$.
We denote the set of components of $X$, admissible for a fixed $y\in V_Y,$ by
\[ \mathcal{C}(X)_{y}=\{C\in\mathcal{C}(X)\ :\  k_C(y)>0\}\]
and its size by $c(X)_{y}.$
}
\end{defn}
Observe that no ambiguity arises between the definition above and Definition \ref{phicon2}, because the indices are taken in different sets.
\begin{lem}\label{rough} Let  $\varphi\in \mathrm{Hom}(X,Y)$, and let $\mathcal{C}(Y)=\{C_i': i\in\{1,\dots,c(Y)\}\}.$ Then \begin{equation}\label{rough-for}
c(X)=\sum_{i=1}^{c(Y)}c(X)_{C'_i}.
\end{equation}
\end{lem}
\begin{proof} Define the  map
 $\varphi_{ \mathcal{C}(X)}: \mathcal{C}(X)\rightarrow \mathcal{C}(Y)$ by $  \varphi_{ \mathcal{C}(X)}(C)=C'$ for all $C\in \mathcal{C}(X),$
 where $C'$ is the unique component of $Y$ with $\varphi(C)\subseteq C'$.
Then $\mathcal{C}(X) _{C_i'}=\varphi_{ \mathcal{C}(X)}^{-1}(C_i')$, for $i\in \{1,\dots,c(Y)\}$. Thus   $ \mathcal{C}(X)=\bigcup_{i=1}^{c(Y)}\mathcal{C}(X) _{C_i'}$ and, since the union is disjoint, we get the desired equality.
  \end{proof}

  \subsection{\bf Counting the components for locally surjective homomorphisms }\label{comp-Lsur}
Formula \eqref{rough-for} is generally of little help in computing $c(X)$ from $c(Y)$ since the numbers $c(X)_{C'_i}$ are hard to determine.
If $\varphi$ is locally surjective,
by Proposition \ref{hom-con}, we have $\mathcal{C}(X) _{C'}=\{C\in  \mathcal{C}(X)\ :\  \varphi(C)=C' \}$ and we can write a more expressive formula.

\begin{lem}\label{second} Let $\varphi\in \mathrm{LSur}(X,Y)$.
 \begin{itemize} \item[(i)]  For each $y\in V_Y,\
\mathcal{C}(X) _{C_{Y}(y)}=\mathcal{C}(X)_{y}.$
In particular $c(X) _{C_{Y}(y)}=c(X)_{y}.$
 \item[(ii)]  If $y, \overline{y}\in V_{C'}$, for some $C'\in \mathcal{C}(Y)$, then $c(X)_{y}=c(X)_{\overline{y}}.$
 \item[(iii)] For $1\leq i\leq c(Y)$, let $y_i\in V_Y$ be such that $\mathcal{C}(Y)=\{C_Y(y_i): 1\leq i\leq c(Y)\}$. Then
 \begin{equation}\label{genfor}
 c(X)=\sum_{i=1}^{c(Y)}c(X)_{y_i}.
 \end{equation}
 \end{itemize}
\end{lem}
\begin{proof}  (i) Let $C\in \mathcal{C}(X) _{C_{Y}(y)}$. Thus, as $\varphi\in \mathrm{LSur}(X,Y)$, $\varphi(C)=C_{Y}(y)$ so that, in particular, there exists $x\in V_C$ with $\varphi(x)=y$ and so $k_C(y)>0$. Conversely if $C\in\mathcal{C}(X)$ and $ k_C(y)>0,$ then there exists  $x\in V_C$ with $\varphi(x)=y$. Thus $C=C_{X}(x)$ and, by Proposition \ref{hom-con}, $\varphi(C)=\varphi(C_{X}(x))=C_{Y}(\varphi(x))=C_{Y}(y)$.

(ii)-(iii) They follow immediately as an application of (i) and of Lemma \ref{rough}.
\end{proof}
Note that the integers $c(X)_{y_i}$ in \eqref{genfor} are non-negative, and that $c(X)_{y_i}=0$ if and only if the component $C_i'$ is not included in the image of $X$ by $\varphi$.

\begin{proof}[Proof of Theorem A]
Combine Proposition \ref{hom-con}  and Lemma \ref{second} (iii).
\end{proof}
\subsection {\bf The component equitable homomorphisms }\label{equi}

While Formula \ref{genfor} improves Formula \ref{rough-for} allowing to pass from $C_i'$ to one of its vertices $y_i, $ the computation of  $c(X)_{C'_i}$ often remains challenging. Fortunately, in many applications,   we have the following property: every component of $X$ admissible for $y\in V_Y$ intersects the fibre $\varphi^{-1}(y)$ in sets of the same size.
\begin{defn}\label{comp-eq}
 {\rm $\varphi\in \mathrm{Hom}(X,Y)$ is called {\it component equitable} if for every $y\in V_Y$ and every $C,\hat{C}\in \mathcal{C}(X)_{y}$, we have  $k_{C}(y)=k_{\hat{C}}(y)$.
  We denote the set of the component equitable homomorphisms from $X$ to $Y$  by $\mathrm{CE}(X,Y)$. }
\end{defn}

\noindent We exhibit examples showing that generally, among the classes $\mathrm{CE}(X,Y), \mathrm{E}(X,Y), \mathrm{PC}(X,Y)$, no further inclusion apart from  $\mathrm{E}(X,Y)\cap \mathrm{Com}(X,Y)\subseteq \mathrm{PC}(X,Y)$, proved in \eqref{sub2}, holds. In all of the following examples let $Y$ be defined by $V_Y=\{y,z\}$, $E^*_Y=\{y,z\}$.

\begin{ex}\label{PC+CE-E} {\rm Let $X$ be defined by $$V_X=\{1,2,3,4,5,6,7,8\},\quad E^*_X=\{\{1,2\},\{1,5\},\{1,6\},\{2,6\},\{3,4\},\{3,7\},\{4,8\}\}$$ and consider $\varphi:V_X\rightarrow V_Y$ given by $\varphi(x)=y$ for all $1\leq x\leq 4$, $\varphi(x)=z$ for all $5\leq x\leq 8.$
Then $\varphi\in (\mathrm{PC}(X,Y)\cap \mathrm{CE}(X,Y))\setminus \mathrm{E}(X,Y)$.}
\end{ex}

\begin{ex}\label{PC-E-CE} {\rm  Let $X$ be defined by $$V_X=\{1,2,3,4,5,6\},\quad E^*_X=\{\{1,2\}, \{1,4\},\{2,5\}, \{3,6\}\}$$ and consider $\varphi:V_X\rightarrow V_Y$ given by $\varphi(x)=y$ for all $1\leq x\leq 3$, $\varphi(x)=z$ for all $4\leq x\leq 6.$
Then $\varphi\in \mathrm{PC}(X,Y)\setminus\left( \mathrm{CE}(X,Y)\cup\mathrm{E}(X,Y)\right)$.}
\end{ex}

\begin{ex}\label{E+Com-CE}{\rm   Let $X$ be defined by $V_X=\{x\in\mathbb{N}: 1\leq x\leq 14\}$ and
$$
 \begin{array}{l} E^*_X=\{\{1,2\}, \{1,3\},\{1,8\}, \{2,3\},\{2,9\},\{3,10\},\{4,5\},\{4,7\},\{4,11\},\{5,6\},\{5,12\},\\
\{6,7\},\{6,13\},\{7,14\},\{8,9\},\{8,10\},\{9,10\},\{11,12\},\{11,14\},\{12,13\}, \{13,14\}\}.
\end{array}
$$
Consider $\varphi:V_X\rightarrow V_Y$ given by $\varphi(x)=y$ for all $1\leq x\leq 7$, $\varphi(x)=z$ for all $8\leq x\leq 14.$
Then $\varphi\in (\mathrm{E}(X,Y)\cap\mathrm{Com}(X,Y) )\setminus \mathrm{CE}(X,Y)$.}
\end{ex}

\begin{prop}\label{union2} Let $\varphi\in \mathrm{LSur}(X,Y)\cap \mathrm{CE}(X,Y)$ and $y\in V_Y$. Then
$c(X)_{y}=\frac{k_{X}(y)}{k_{C}(y)}$ for all $C\in \mathcal{C}(X)_{y}.$ In particular $k_{C}(y)$ divides  $k_{X}(y).$
\end{prop}
\begin{proof}
$c(X)_{y}$ is the number of components  of $X$ admissible for $y$ and, since $\varphi\in \mathrm{CE}(X,Y)$, each of those components admits the same number of vertices mapped by $\varphi$ into $y$. Thus, for each $C \in \mathcal{C}(X)_{y}$, we have $c(X)_{y} k_{C}(y)=k_{X}(y)$, where the factors are positive integers.
\end{proof}

\subsection{\bf Counting the components for orbit homomorphisms }\label{automorphic}

\begin{prop}\label{union2cons} Let $\varphi\in \mathrm{O}(X,Y)$ be $\mathfrak{G}$-consistent and let $y\in V_Y$.

\begin{itemize}
\item[(i)] For each $C\in \mathcal{C}(X)_{y},\  \mathcal{C}(X)_{y}=\{f(C)\ : f\in \mathfrak{G}\}.$ In particular, the components  of $X$ admissible for $y$ are isomorphic through a graph automorphism of $X$.
\item[(ii)] $\mathrm{O}(X,Y)\subseteq \mathrm{CE}(X,Y).$
\item[(iii)] If $\varphi\in \mathrm{O}(X,Y)\cap  \mathrm{Com}(X,Y)$, then $c(X)_{y}=\frac{|\varphi^{-1}(\varphi(V_C))|}{|V_C|}$ for all $C\in \mathcal{C}(X)_{y}.$
\end{itemize}
\end{prop}
\begin{proof} (i) Let $C\in \mathcal{C}(X)_{y}$ and $f\in \mathfrak{G}$.  Then there exists $x\in V_X$ such that $\varphi(x)=y$ and, as
 $f\in \mathrm{Aut}(X),$ we have $f(C)=f(C_X(x))=C_Y(f(x))$. By
condition (a) in Lemma \ref{phi-con}, we have that $\varphi\circ f=\varphi.$ Thus
$\varphi(f(x))=\varphi(x)=y$ which gives $f(C)\in \mathcal{C}(X)_{y}$. So $\{f(C)\ : f\in \mathfrak{G}\}\subseteq  \mathcal{C}(X)_{y}.$ Note also that $f(\varphi^{-1}(y)\cap V_C)= \varphi^{-1}(y)\cap V_{f(C)}$.  Since $f$ is a bijection, that implies
\begin{equation} \label{A} k_{f(C)}(y)=k_{C}(y).
\end{equation}
We next show $\mathcal{C}(\Gamma)_{y}\subseteq \{f(C): f\in \mathfrak{G}\}.$ Let $\hat{C}\in \mathcal{C}(X)_{y}$ and let $\hat{x}\in V_{\hat{C}}$ such that $\varphi(\hat{x})=y$.  Thus we have $\varphi(x)=\varphi(\hat{x})$ and, by condition (b) in Lemma \ref{phi-con}, there exists $f\in  \mathfrak{G}$ with $\hat{x}=f(x).$ It follows that $\hat{x}\in f(V_C)=V_{f(C)}.$ Hence $\hat{C}$ and $f(C)$ are components with a vertex in common, which implies that $\hat{C}=f(C)$.

(ii) Use (i) and \eqref{A}.

(iii) Let $C\in \mathcal{C}(X)_{y}$ and $x\in V_C$ with $\varphi(x)=y$.
By (i), all the components in $\mathcal{C}(X)_{y}$ have the same number of vertices, so that, $c(X)_{y} |V_C|$ counts the vertices of all the components in $\mathcal{C}(X)_{y}$, that is, the size of the set  $\bigcup_{\hat{C}\in  \mathcal{C}(X)_{y}} V_{\hat{C}}$.  By \eqref{sub2}, we have $\varphi\in \mathrm{O}(X,Y)\cap \mathrm{Com}(X,Y)\subseteq\mathrm{LSur}(X,Y)$. Thus we can apply Proposition \ref{hom-con} (ii) to $\varphi,$ obtaining
$C_{Y}(y)=\varphi(C).$ So, by Lemma \ref{second} (i), we get $\mathcal{C}(X)_{y}=\mathcal{C}(X)_{C_{Y}(y)}=\mathcal{C}(X)_{\varphi(C)}$. Hence, by Proposition \ref{hom-con} (iii),  \[\displaystyle{\bigcup_{\hat{C}\in  \mathcal{C}(X)_{y}} V_{\hat{C}}}=\displaystyle{\bigcup_{\hat{C}\in  \mathcal{C}(X)_{\varphi(C)}}V_{\hat{C}}}=\varphi^{-1}(\varphi(V_C)),\]
which gives $c(X)_{y} |V_C|=|\varphi^{-1}(\varphi(V_C))|.$
\end{proof}

\begin{proof}[Proof of Theorem B] By Propositions \ref{propc} and \ref{union2cons} we have
\begin{equation*}
 \mathrm{O}(X,Y)\cap\mathrm{Com}(X,Y)\subseteq  \mathrm{CE}(X,Y)\cap \mathrm{LSur}(X,Y)\cap \mathrm{Com}(X,Y).
 \end{equation*}
 Thus the assertion follows, combining Lemma \ref{second} and Proposition \ref{union2}.
\end{proof}

Note that Formula \eqref{best-for} is more manageable than Formula \eqref{rough-for} due to its high level of symmetry. Moreover the terms in the summand are easily computable in many contexts. A remarkable case is given when $X$ is the quotient proper power graph  and $Y$ is the proper power type graph of  a fusion controlled permutation group. That case will be examined in \cite{BIS1} and  \cite{BIS2}. We now write an explicit procedure for computing $c(X)$ based upon our results.

\begin{procedure} {\bf Procedure to compute  $c(X)$ for $\varphi\in  \mathrm{O}(X,Y)\cap\mathrm{Com}(X,Y)$ } \label{procedure}
\begin{itemize}\item[ (I)] {\it  Selection of $y_i$ and $C_i.$}
\end{itemize}
\begin{itemize}
\item[ {\it Start}]  : Pick arbitrary $y_1\in V_Y$ and choose any $C_1\in\mathcal{C}(X)_{y_1}$.

\item[ {\it Basic step}]: Given $y_1,\dots, y_i\in V_Y$ and $C_1,\dots,C_i\in\mathcal{C}(X)$ such that $C_j\in \mathcal{C}(X)_{y_j}\ (1\leq j\leq i)$, choose any $y_{i+1}\in V_Y\setminus \bigcup_{j=1}^iV_{C_{Y}(y_j)}=V_Y\setminus \bigcup_{j=1}^iV_{\varphi(C_j)}$ and  any $C_{i+1}\in \mathcal{C}(X)_{y_{i+1}}.$

\item[ {\it Stop}]: The procedure stops in $l=c(Y)$ steps.
\end{itemize}
\medskip

\begin{itemize}\item[ (II)] {\it  The value of  $c(X)$. }
\end{itemize}
 Compute the integers $\frac{k_{X}(y_j)}{k_{C_j}(y_j)}\  (1\leq j\leq c(Y))$ and sum them up to get  $c(X)$.
 \end{procedure}
\bigskip

Given a graph $X$, Procedure 6.10 may be applied to any graph $Y$ such that $ \mathrm{O}(X,Y)\cap\mathrm{Com}(X,Y)\neq \varnothing$ once $\varphi\in \mathrm{O}(X,Y)\cap\mathrm{Com}(X,Y)$ is chosen. Such $Y$, as explained in Section \ref{equi-orb}, are the quotients of $X$ with respect to the orbit partitions of the possible $\mathfrak{G}\leq \mathrm{Aut}(X)$, and $\varphi$ are the corresponding projection on $Y.$ Choices of $\mathfrak{G}$ with different sets of orbits
 lead to different computations of the coefficients $\frac{k_{X}(y_j)}{k_{C_j}(y_j)}$, with the computation easier when $\mathfrak{G}$ is ``large''.

\vskip 0.6 true cm
 \section{\bf The isomorphism class of the components }\label{iso-class}
\vskip 0.4 true cm
Under the assumption $\varphi\in \mathrm{PC}(X,Y)$, Proposition \ref{hom-con} guarantees that each component $C$ of $X$ admits as quotient the component $\varphi(C)$ of $Y$.
In this short section, we  study when $C$ is actually isomorphic to $\varphi(C)$.

\begin{lem}\label{iso} Let $\varphi\in \mathrm{PC}(X,Y)$. Given $C\in \mathcal{C}(X),$ we have $C\cong \varphi(C)$ if and only if $k_C(y)=1$ for all $y\in \varphi(C).$
 \end{lem}
\begin{proof} Since $\varphi_{\mid C}:C\rightarrow \varphi(C)$ is always a complete homomorphism, $\varphi_{\mid C}$ is an isomorphism if and only if it is injective, that is,
$k_C(y)=|V_C\cap \varphi^{-1}(y)|=1$ for all $y\in \varphi(C).$
\end{proof}

\begin{prop} \label{independence}  Let $\varphi\in \mathrm{O}(X,Y)\cap \mathrm{Com}(X,Y)$ and let $C\in \mathcal{C}(X)$.
\begin{itemize}
\item[(i)] If $y,\overline{y}\in V_{\varphi(C)}$, then $\frac{k_{X}(y)}{k_{C}(y)}=\frac{k_{X}(\overline{y})}{k_{C}(\overline{y})}.$
\item[(ii)] $C\cong \varphi(C)$ if and only if
there exists $y\in V_{\varphi(C)}$ such that $k_{C}(y)=1$ and, for every $\overline{y}\in V_{\varphi(C)}$, $k_{X}(y)=k_{X}(\overline{y}).$
\item[(iii)] If there exists $y\in V_{\varphi(C)}$ such that $k_{C}(y)=k_{X}(y),$ then for every $\overline{y}\in V_{\varphi(C)}$, $k_{C}(\overline{y})=k_{X}(\overline{y}).$
\item[(iv)] If there exists $y\in V_{\varphi(C)}$ such that $k_{C}(y)=k_{X}(y)>1,$ then $C\not\cong \varphi(C).$
\end{itemize}
\end{prop}

\begin{proof}(i)  Since $C\in \mathcal{C}(X)_{y}\cup  \mathcal{C}(X)_{\overline{y}},$ by Proposition \ref{union2}, we get $c(X)_{y}=\frac{k_{X}(y)}{k_{C}(y)}$ as well as $c(X)_{\overline{y}}=\frac{k_{X}(\overline{y})}{k_{C}(\overline{y})}.$ Now, by \eqref{sub2}, $\mathrm{O}(X,Y)\cap \mathrm{Com}(X,Y)\subseteq \mathrm{LSur}(X,Y)\cap \mathrm{Com}(X,Y)$. Thus Lemma \ref{second} (ii) applies giving $c(X)_{y}=c(X)_{\overline{y}}$ and the equality follows.

(ii)-(iv)  They are immediate from (i) and Lemma \ref{iso}.
\end{proof}

\bigskip

{\bf Acknowledgements} The author wishes to thank Gena Hahn, Cheryl Praeger, Ji\v{r}\'i Fiala, J\"urgen Lerner and Ulrich Knauer for helpful comments on a preliminary version of the paper.  A particular thank to the anonymous referee whose advices led to a significative improvement of the paper.
The author is partially supported by GNSAGA of INdAM.

\bigskip
\bigskip

{\footnotesize \pn{\bf Daniela~Bubboloni}\; \\  {Dipartimento di Matematica e Informatica U.Dini},\\
{ Viale Morgagni 67/a}, 50134 Firenze, Italy}\\
{\tt Email:  daniela.bubboloni@unifi.it}\\

\end{document}